\documentclass[a4paper,reqno,11pt]{amsart}
\usepackage{amssymb}
\usepackage{amstext}
\usepackage{amsmath}
\usepackage{amscd}
\usepackage{amsthm}
\usepackage{amsfonts}
\usepackage{enumerate}
\usepackage{graphicx}
\usepackage{latexsym}
\usepackage{mathrsfs}
\input xy
\xyoption{all}
\usepackage{pstricks}
\usepackage{lscape}
\usepackage{comment}

\newtheorem{theorem}{Theorem}[section]

\newtheorem{corollary}[theorem]{Corollary}
\newtheorem{lemma}[theorem]{Lemma}
\newtheorem{definition-theorem}[theorem]{Definition-Theorem}
\newtheorem{proposition}[theorem]{Proposition}

\theoremstyle{definition}
\newtheorem{definition}[theorem]{Definition}
\newtheorem{remark}[theorem]{Remark}

\newtheorem*{example*}{Example}

\numberwithin{equation}{section}
\newcommand{\C}{{\mathsf C}}
\newcommand{\CC}{\mathsf{C}}
\newcommand{\D}{{\mathsf D}}

\newcommand{\KK}{\mathsf{K}}

\newcommand{\T}{{\mathsf T}}

\newcommand{\add}{\mathsf{add}\hspace{.01in}}

\newcommand{\Kb}{\mathsf{K}^{\rm b}}

\newcommand{\Dfd}{\mathsf{D}_{\rm fd}}

\newcommand{\proj}{\mathsf{proj}\hspace{.01in}}
\newcommand{\per}{\mathsf{per}\hspace{.01in}}

\newcommand{\thick}{\mathsf{thick}\hspace{.01in}}

\newcommand{\dctilt}{{\it d}\strut\kern-.2em\operatorname{-ctilt}\nolimits}
\newcommand{\Ga}{\Gamma}
\newcommand{\tw}{{\rm tw}}
\newcommand{\TW}{{\rm TW}}

\newcommand{\id}{{\rm id}}
\newcommand{\fd}{\mathsf{fd}}
\newcommand{\dg}{\mathsf{dg}}
\newcommand{\Lotimes}{\mathop{{\otimes}^\mathbf{L}_\Ga}\nolimits}
\newcommand{\Lotimespi}{\mathop{{\otimes}^\mathbf{L}_\Pi}\nolimits}
\newcommand{\Le}{\mathsf{L}}
\newcommand{\Ri}{\mathsf{R}}
\newcommand{\tr}{\mathrm{tr}}
\newcommand{\ten}{\otimes}

\newcommand{\Hom}{\operatorname{Hom}\nolimits}
\newcommand{\RHom}{\mathbf{R}\strut\kern-.2em\operatorname{Hom}\nolimits}
\newcommand{\HHom}{\mathcal{H}\strut\kern-.2em\operatorname{om}\nolimits}

\newcommand{\silt}{\mbox{\rm silt}\hspace{.01in}}

\newcommand{\xto}{\xrightarrow}

\begin{document}
\title{derived preprojective algebras and spherical twist functors}

\author{Yuya Mizuno}
\address{Faculty of Liberal Arts, Sciences and Global Education, Osaka Metropolitan University, 1-1 Gakuen-cho, Naka-ku, Sakai, Osaka 599-8531, Japan}
\email{yuya.mizuno@omu.ac.jp}

\author{Dong Yang}
\address{D. Yang: School of Mathematics, Nanjing University, 22 Hankou Road, Nanjing 210093, P. R.
China}
\email{yangdong@nju.edu.cn}

\date{\today}
%\subjclass[2100]{Primary~16G20; Secondary~16G60}
\begin{abstract}
We study silting objects over derived preprojective algebras of acyclic quivers by giving a direct relationship between silting objects, spherical twist functors and mutations. Especially  
for a Dynkin quiver, we establish a bijection between the elements of the braid group and the set of isomorphism classes of basic silting objects over the derived preprojective algebra. 
\end{abstract}
\maketitle
%\tableofcontents
%%%%%%%%%%%%%%%%%%%%%%%%%%%%%%%%%%%%%%%%%%%%%%%%%%%%%%%%%%%%%%%%%%%%%%%%%%%%%%%%%%%%%%%%%%%%%%%%%

\section{Introduction}
Preprojective algebras of quivers form one of the most fundamental and important classes of algebras, and 
their representation theory has been extensively studied for a long time.
Among others, 
the recent developments of $\tau$-tilting theory and silting theory have provided a new method to study 
their derived categories (for example \cite{AM15,IR,BIRS09,M14,M21}). 
These algebras turn out to the 0-th cohomology of their derived version, namely, the (2-)derived preprojective algebras introduced in \cite{Ke2}, which are more symmetric, in the sense that their derived categories are 2-Calabi--Yau triangulated categories. There are two cases. If the quiver $Q$ is non-Dynkin, then the (usual) 
preprojetive algebra of $Q$ is quasi-isomorphic to the derived preprojective algebra of $Q$ and hence their derived categories are triangle equivalent. If the quiver $Q$ is Dynkin, then the derived preprojective algebra has non-trivial cohomology in each non-positive degree, so the preprojective algebra and the derived preprojective algebra are completely different. 
One of the remarkable difference is the existence of \emph{spherical objects} in the case of the derived preprojective algebra, 
while the usual preprojective algebra is selfinjective and their projective resolutions are periodic, so that there is no spherical objects. 
Recall that spherical objects and spherical twist functors were introduced by 
Seidel--Thomas \cite{ST01} and they provide  auto-equivalences of triangulated
categories. One of the important results of \cite{ST01} is the braid group relation of the spherical twist functors and their faithful action in type $\mathbb{A}$.
 
The aim of this paper is to study silting objects of derived preprojective algebras by studying the interaction between silting theory and derived equivalences. More precisely, we 
give an explicit description of spherical twist functors as derived tensor functors of silting complexes and explain the braid group action on the functors.  
Let $Q$ be an acyclic quiver. Let $B_Q$ be the braid group associated to $Q$ with generators $a_i$, where $i$ runs over all vertices of $Q$, and $\Gamma:=\Gamma(Q)$ be the derived preprojective algebra of $Q$ (see Section \ref{ss:derived-preprojective-algebra}). 
For a vertex $i$ of $Q$, let $e_i$  and $t_i$ be the corresponding trivial path and the loop of degree $-1$, respectively. Let 
\[
I_i:=I_i^+=\Gamma(1-e_i)\Gamma+\Gamma t_i\Gamma,
\] 
and $I_i^-:=\Hom_{\C_{\dg}(\Gamma)}(I_i,\Gamma)$, which is a dg $\Gamma$-$\Gamma$-bimodule, where $\C_{\dg}(\Gamma)$ is the dg category of right dg $\Gamma$-modules.  
We denote by $\tw_{S_i}^-(-)$ the dual twist functor (Section~\ref{sss:twist-funtors}), where $S_i=\Gamma/I_i$. The following theorem is our main result, which establishes a relation between spherical twist functors and mutations
as well as a relation between the braid group elements and silting objects over $\Gamma$.

\begin{theorem}[{Theorem~\ref{main1}, Lemma~\ref{I=mu}, Theorem~\ref{main2} and Corollary~\ref{cor:braid-group-action}}]
\label{thm:main}
\mbox{}
\begin{itemize}
\item[(1)] %Let $i$ be a vertex of $Q$. 
There is an isomorphism of triangle autoequivalences of the derived category $\D(\Gamma)$
$$-\Lotimes I_i\cong\tw_{S_i}^{-}(-),$$
which restrict to isomorphic autoequivalences of $\per(\Gamma)$.
\item[(2)]
There is an isomorphism in the perfect derived category $\per(\Ga)$ 
\[
I_{i_1}^{\epsilon_{1}}\ten_\Gamma I_{i_2}^{\epsilon_{2}}\ten_\Gamma\cdots \ten_\Gamma I_{i_k}^{\epsilon_{k}}\cong \mu_{i_1}^{\epsilon_{1}}\circ\mu_{i_2}^{\epsilon_2}\cdots \circ \mu_{i_k}^{\epsilon_{k}}(\Ga),
\]
for any sequence $(\epsilon_1,\ldots,\epsilon_k)$ of signs, where $\mu^+$ (respectively, $\mu^-$) is the left (respectively, right) mutation. 
\item[(3)] The assignment $a= a_{i_1}^{\epsilon_{1}}a_{i_2}^{\epsilon_2}\cdots a_{i_k}^{\epsilon_{k}}\mapsto I_{i_1}^{\epsilon_{1}}\ten_\Gamma I_{i_2}^{\epsilon_{2}}\ten_\Gamma\cdots \ten_\Gamma I_{i_k}^{\epsilon_{k}}$
defines a map from the braid group $B_Q$ to the set of equivalence classes of silting objects of $\per(\Gamma)$. It is bijective when $Q$ is Dynkin.
\item[(4)] The braid group $B_Q$ acts on $\D(\Gamma)$. When $Q$ is Dynkin, this action is faithful.
\end{itemize}
\end{theorem}

In Section~\ref{s:preminimaries}, we will recall the relevant definitions and basic results. In Section~\ref{s:main-results} we will prove Theorem~\ref{thm:main}.

%%%%%%%%%%%%%%%%%%%%%%%%%%%%%%%%%%%%%%%%%%%%%%%%%%%%%%%%%%%%%%%%%%%%%%%%%%%%%%%%%%%%%%%%%%%%%%%%%
\subsection*{Notation and convention}
Throughout this paper, let $K$ be a field and let $D=\Hom_K(-,K)$ denote the $K$-dual. 
Let $\T$ be a $K$-linear additive category. We say that $\T$ is \emph{Hom-finite} if all morphism spaces of $\T$ are finite-dimensional over $K$. For an object $M$ of $\T$, we denote by $\add(M)$ the smallest full subcategory of $\T$ which contains $M$ and which is closed under taking finite direct sums and direct summands. %For a full subcategory $\X$ of $\T$, define full subcategories $\X^{\perp_{\T}}:=\X^\perp$ and ${}^{\perp_{\T}}\X:={}^{\perp}\X$ of $\T$ as 
%\begin{align*}\X^\perp&:=\{Y\in\T\mid \Hom_\T(X,Y)=0~\forall X\in\X\}\\{}^\perp\X&:=\{Y\in\T\mid \Hom_\T(Y,X)=0~\forall X\in\X\}.\end{align*}

Assume that $\T$ is triangulated with shift functor $[1]$. For an object $M$ of $\T$, we denote by $\thick(M)$ the \emph{thick subcategory} of $\T$ generated by $M$, that is, the smallest triangulated subcategory of $\T$ which contains $M$ and which is closed under taking direct summands. %For two full subcategories $\X$ and $\Y$ of $\T$, define $\X*\Y$ as the full subcategory of $\T$ consisting of objects $M$ which admits a triangle $X\to M\to Y\to X[1]$ with $X\in\X$ and $Y\in\Y$. 

%Let $\Lambda$ be a finite-dimensional $K$-algebra. We denote by $\mod\Lambda$ the category of finite-dimensional right modules over $\Lambda$ and by $\D^b(\mod\Lambda)$ the bounded derived category of $\mod\Lambda$. We denote by $\proj\Lambda$ the category of finite-dimensional right projective modules over $\Lambda$ and by $\KK^b(\proj\Lambda)$ the bounded homotopy category of $\proj\Lambda$.

\subsection*{Acknowledgement}
The authors thank Bernhard Keller and Martin Kalck for informing them of the results in \cite{BapatDeopurkarLicata23} and \cite{HaraWemyss22}. They thank a referee for very carefully reading the manuscript and for giving them valuable comments. The first author is supported by 
Grant-in-Aid for Scientific Research (C) 20K03539. 
The second author acknowledges support by the National Natural Science Foundation of China No. 12031007 and No. 11401297.

%Let $K$ be a field.
%Throughout this paper, a triangulated category $\T$ is assumed to be $K$-linear Hom-finite Krull-Schmidt category with the shift functor $[1]$ unless we state otherwise. 
%All subcategories are strictly full.
%
%Let $\X,\Y$ be subcategories of $\T$. 
%Then $\Hom_{\T}(\X,\Y)=0$ means $\Hom_{\T}(X,Y)=0$ for all objects $X\in \X$ and $Y\in Y$.
%$X\ast Y:=\add X\ast \add Y$.
%\comment{$\mod \Lambda$}
%${}^{\perp}M$ and $M^\perp$.
%We denote by $\add M$ ... and by $\thick M$ ....
%All shift functor is $[1]$
%%%%%%%%%%%%%%%%%%%%%%%%%%%%%%%%%%%%%%%%%%%%%%%%%%%%%%%%%%%%%%%%%%%%%%%%%%%%%%%%%%%%%%%%%%%%%%%%%

\section{Preliminaries}
\label{s:preminimaries}

The aim of this section is to recall the definitions and some basic results on silting objects, silting mutation, derived categories of dg algebras, spherical twist functors and derived preprojective algebras. 
%\subsection{Krull--Schmidt categories}We say that $\T$ is \emph{Hom-finite} if all morphism spaces of $\T$ are finite-dimensional over $K$. %We say that $\T$ is  \emph{idempotent complete} if any idempotent $e:M\to M$ in $\T$ arises from a direct sum decomposition $M\cong \mathrm{im}(e)\oplus \mathrm{ker}(e)$. We say that $\T$ is \emph{Krull--Schmidt} if every object $M$ of $\T$ is isomorphic to $M_1\oplus\ldots \oplus M_n$ with $M_1,\ldots,M_n$ having local endomorphism algebras. %If $\T$ is Hom-finite, then $\T$ is Krull--Schmidt if and only if $\T$ is idempotent complete. The object $M$ is said to be \emph{basic} if in the above decomposition $M_1\oplus\ldots\oplus M_n$ the objects $M_1,\ldots,M_n$ are pairwise non-isomorphic. In this case, we denote $|M|=n$.
%%%%%%%%%%%%%%%%%%%%%%%%%%%%%%%%%%%%%%%%%%%%%%%%%%%%%%%%%%%%%%%%%%%%%%%%%%%%%%%%%%%%%%%%%%%%%%%%%

\subsection{Silting objects}
\label{sectionsilting}
\label{ss:silting}
Let $\T$ be a $K$-linear triangulated category with shift functor $[1]$. 
We recall the definition of silting objects and silting mutations from \cite{AI12}.

\begin{definition}
\begin{itemize}
\item[(1)] An object $M$ of $\T$ is said to be \emph{presilting} if $\Hom_{\T}(M,M[p])=0$ for all positive integers $p$.
\item[(2)] A presilting object $M$ of $\T$ is said to be \emph{silting} if $\T=\thick(M)$. 
\item[(3)] Two silting objects $M$ and $N$ are said to be \emph{equivalent} if $\add(M)=\add(N)$. We denote by $\silt\T$ the set of equivalence classes of silting objects of $\T$. 
\end{itemize}
\end{definition}

Here is an example: 
Let $\Lambda$ be a $K$-algebra. Then $\Lambda$ is a silting object of the bounded homotopy category $\Kb(\proj\Lambda)$ of finitely generated projective right $\Lambda$-modules. 

%We assume that $\T$ is Krull--Schmidt (i.e., every object $M$ of $\T$ is isomorphic to $M_1\oplus\ldots \oplus M_n$ with $M_1,\ldots,M_n$ having local endomorphism algebras) and Hom-finite (i.e. all morphism spaces of $\T$ are finite-dimensional). 

For a decomposition $M=X\oplus Y$ such that $\add(X)\cap\add(Y)=0$, if there is a triangle
\begin{align}
X\xto{f}Y'\xto{}X^{\ast}\xto{}X[1] \notag
\end{align}
with a left $\add Y$-approximation $f$ of $X$, then we call $\mu_{X}^{\Le}(M):=X^{\ast}\oplus Y$ a \emph{left mutation} of $M$ with respect to $X$.  By \cite[Theorem 2.31]{AI12}, a left mutation of a silting object is again a silting object. It is well-defined on equivalence classes of silting objects. If $\T$ is Krull--Schmidt and Hom-finite, we usually take $f$ to be minimal so that if $M$ is basic then so is $\mu_X^L(M)$ and we can talk about \emph{the} left mutation.
Dually, we define a \emph{right mutation} $\mu_{X}^{\Ri}(M)$. In this paper, by \emph{mutation} will mean left or right mutation.  If $X$ is indecomposable, we call the mutation an \emph{irreducible mutation}.

For objects $M,N$ of $\T$, we write $M\ge N$ if $\Hom_{\T}(M,N[n])=0$ for all positive integers $n$.
Then the relation $\ge$ gives a partial order on $\silt\T$ by \cite[Theorem 2.11]{AI12}. If $N$ is obtained from a silting object $M$ by a sequence of left mutations, then we have $M>N$ \cite[Theorem 2.35]{AI12}.

%Next we recall the notion of silting-discrete triangulated categories.
%\begin{definition}\label{def silting-discrete}A triangulated category $\T$ is said to be \emph{silting-discrete} if, for any basic silting object $M$, the set $\ssilt{n_{M}}\T$ is finite for any  positive integer $n$.\end{definition}

%Moreover, if $\T$ is silting-discrete, then we can obtain all silting objects in $\T$ from a silting object by a finite sequence of mutations (see \cite[Corollary 3.9]{Ai13}). By the result of \cite{AM15}, we have a criterion for a triangulated category to be silting-discrete.

%\begin{lemma}[{\cite[Theorem 2.4]{AM15}}]\label{AM} The category $\T$ is silting-discrete if and only if, for an arbitrary basic silting object $M$ of $\T$, the set $\ssilt{2_{M}}\T$ is finite.\end{lemma}

%%%%%%%%%%%%%%%%%%%%%%%%%%%%%%%%%%%%%%%%%%%%%%%%%%%%%%%%
%%%%%%%%%%%%%%%%%%%%%%%%%%%%%%%%%%%%%%%%%%%%%%%%%%%%%%%%%%%%%%%%%%%%%%%%%%%%%%%%%%%%%%%%%%%%%%%%%%%%%%%%%%%%%%%%%%%%%%%%%%%%%%%%%%%%%%%%%%%%%%%%
\subsection{Derived categories of dg algebras}
Let $A$ be a differential graded (dg for short) $K$-algebra. Let $\C_{\dg}(A)$ be the dg category of right dg $A$-modules, see \cite[Section 3.1]{Ke4}. For $M,N\in\C_{\dg}(A)$, the Hom-complex $\Hom_{\C_{\dg}(A)}(M,N)$ in $\C_{\dg}(A)$ is the complex of $K$-vector spaces with its degree $p$ component $\Hom^p_{\C_{\dg}(A)}(M,N)$ being the space of homogeneous $A$-linear maps of degree $p$ from $M$ to $N$ (here we consider $A$ as a graded algebra and $M,N$ as graded $A$-modules) and with its differential defined by $d(f)=d_N\circ f-(-1)^p f\circ d_M$ for $f\in\Hom^p_{\C_{\dg}(A)}(M,N)$. For example, $\Hom_{\C_{\dg}(A)}(A,A)=A$. A right dg $A$-module $M$ is said to be \emph{$\KK$-projective} if $\Hom_{\C_{\dg}(A)}(M,N)$ is acyclic for any acyclic right dg $A$-module $N$. For example, $A_A$ is $\KK$-projective.  The property of being $\KK$-projective is closed under taking direct sums, direct summands, shifts, and mapping cones of homomorphisms of dg modules. Here, by a homomorphism, we mean a closed morphism of degree $0$.

The dg category $\C_{\dg}(A)$ is a pretriangulated dg category in the sense of \cite[Section 4.3]{Ke4}. Let $\KK(A)$ be the homotopy category of $\C_{\dg}(A)$, see \cite[Section 2.2]{Ke4}. Precisely, the objects of $\KK(A)$ are right dg $A$-modules, and for two right dg $A$-modules $M$ and $N$ we have
\[
\Hom_{\KK(A)}(M,N)=H^0\Hom_{\C_{\dg}(A)}(M,N).
\]
$\KK(A)$ is a triangulated category with shift functor being the shift of dg modules. Let $\mathrm{acyc}(A)$ be the full subcategory of $\KK(A)$ consisting of acyclic right dg $A$-modules. The \emph{derived category} $\D(A)$ of right dg $A$-modules is defined as the triangle quotient of $\KK(A)$ by $\mathrm{acyc}(A)$. For $p\in\mathbb{Z}$ and, for right dg $A$-modules $M$ and $N$ with $M$ being $\KK$-projective, we have
\begin{align}\label{eq:Hom-space-in-derived-category}
\Hom_{\D(A)}(M,N[p])=\Hom_{\KK(A)}(M,N[p])=H^0\Hom_{\C_{\dg}(A)}(M,N[p]).
\end{align}
Let $\per(A)$ be the thick subcategory of $\D(A)$ generated by $A_A$, and $\Dfd(A)$ the full subcategory of $\D(A)$ consisting of right dg $A$-modules whose total cohomology is finite-dimensional over $K$. For dg $K$-algebras $A$ and $B$, a triangle equivalence $F:\D(B)\to\D(A)$ restricts to triangle equivalences $\per(B)\to\per(A)$ and $\Dfd(B)\to\Dfd(A)$, see for example \cite[Lemmas 2.6 and 2.7]{AMY}.

\begin{lemma}\label{lem:non-positive-dg-algebra-and-silting}
If $H^{p}(A)=0$ for all $p>0$, then $A$ is a silting object of $\per(A)$.
\end{lemma}
\begin{proof}
This follows from the formula \eqref{eq:Hom-space-in-derived-category} since $\Hom_{\C_{\dg}(A)}(A,A)=A$.
\end{proof}
%%%%%%%%%%%%%%%%%%%%%%%%%%%%%%%%%%%%%%%%%%%%%%%%%%%%%%%%%%%%%%%%%%%%%%%%%%%%%%%%%%%%%%%%%%%%%%%%%

Let $A$ and $B$ be dg $K$-algebras and $F:\C_{\dg}(B)\to \C_{\dg}(A)$ be a dg functor. The \emph{left derived functor} $\mathbf{L}F:\D(B)\to\D(A)$ of $F$ is the triangle functor $\pi\circ F\circ\mathbf{p}$, where $\mathbf{p}:\D(B)\to\KK(B)$ is the functor of taking $\KK$-projective resolutions (\cite[Section 3]{Ke1}) and $\pi:\KK(A)\to\D(A)$ is the projection (see \cite[Section 6.4]{Ke1}). Moreover, $F(B)$ admits the structure of a left dg $B$-module via the dg algebra homomorphism $F(B,B):B\to\Hom_{\C_{\dg}(A)}(F(B),F(B))$, making $F(B)$ a dg $B$-$A$-bimodule. The following lemma is obtained by applying \cite[Lemma 6.4]{Ke1} to the special case $\mathcal{A}=A$ and $\mathcal{B}=B$.

\begin{lemma}\label{lem:turning-a-derived-functor-into-tensor}
Assume that $\mathbf{L}F$ commutes with infinite direct sums. Then $\mathbf{L}F$ is isomorphic to $-\mathop{{\otimes}^\mathbf{L}_B} X:\D(B)\to \D(A)$, where $X=F(B)$.
\end{lemma}

Let $X$ be a dg $B$-$A$-bimodule. If $X$ as a left dg $B$-module is $\KK$-projective, then $-\ten_B X$ preserves acyclicity and hence induces a triangle functor $-\ten_B X:\D(B)\to \D(A)$. Moreover, if $f: \mathbf{p}M\to M$ is a $\KK$-projective resolution, then $f\ten \id: \mathbf{p}M\ten_B X\to M\ten_B X$ is a quasi-isomorphism, inducing a natural isomorphism $-\mathop{{\otimes}^\mathbf{L}_B} X\to -\ten_B X$. This is part (a) of the following lemma. Part (b) is obtained by applying \cite[Lemma 6.2 b)]{Ke1} to the special case $\mathcal{A}=A$ and $\mathcal{B}=B$.

\begin{lemma} \label{lem:derived-tensor-functor}
Let $X$ be a dg $B$-$A$-bimodule. 
\begin{itemize}
\item[(a)]
If $X$ as a left dg $B$-module is $\KK$-projective, then $-\mathop{{\otimes}^\mathbf{L}_B} X\cong -\ten_B X$.
\item[(b)]
If $X$ as a right dg $A$-module is $\KK$-projective and $-\mathop{{\otimes}^\mathbf{L}_B} X$ is a triangle equivalence, then $-\mathop{{\otimes}^\mathbf{L}_B} \Hom_{\C_{\dg}(A)}(X,A)$ is a quasi-inverse of $-\mathop{{\otimes}^\mathbf{L}_B} X$.
\end{itemize}
\end{lemma}

%%%%%%%%%%%%%%%%%%%%%%%%%%%%%%%%%%%%%%%%%%%%%%%%%%%%%%%%%%%%%%%%%%%%%%%%%%%%%%%%%%%%%%%%%%%%%%%%%
\subsection{Twist functors}\label{sss:twist-funtors}
We follow \cite{AL13} to define twist functors.
Let $A$ be a dg $K$-algebra. Let $E$ be a right dg $A$-module belonging to $\per(A)\cap\D_\fd(A)$.  We consider $E$ as a dg $K$-$A$-bimodule, and define the \emph{twist functor} $\TW_E\colon \CC_{\dg}(A)\to\CC_{\dg}(A)$ and the
\emph{dual twist functor} $\TW^-_E\colon\CC_{\dg}(A)\to\CC_{\dg}(A)$ by 
\begin{align*}
\TW_E(M)&:=\mathrm{Cone}(\Hom_{\C_{\dg}(A)}(E,M)\ten_K E\overset{\rm ev}{\longrightarrow} M),\\
\TW^-_E(M)&:=\mathrm{Cone}(M\overset{\rm coev}{\longrightarrow}\Hom_{\C_{\dg}(K)}(E^{\tr_K},M\otimes_A E^{\tr_K}))[-1],
\end{align*}
where $E^{\tr_K}=\Hom_{\C_{\dg}(K)}(E,K)$.

Let $\tw^-_E=\mathbf{L}\TW^-_E:\D(A)\to\D(A)$ be the left derived functor of $\TW^-_E$ and $\tw_E=\mathbf{R}\TW_E:\D(A)\to \D(A)$ be the right derived functor of $\TW_E$. By \cite[Proposition 5.3]{AL13}, $\tw^-_E$ is left adjoint to $\tw_E$, in particular, it commutes with infinite direct sums. On objects we have
\begin{align*}
\tw^-_E(M)&=\mathrm{Cone}(\mathbf{p}M\overset{\rm coev}{\longrightarrow}\Hom_{\C_{\dg}(K)}(E^{\tr_K},\mathbf{p}M\otimes_A E^{\tr_K}))[-1].
\end{align*}
For $M\in\D(A)$, there is a chain of quasi-isomorphisms due to the assumption $E\in\D_\fd(A)$
\[
M\ten_A E^{\tr_K}\ten_K E\to M\ten_A E^{\tr_K}\ten_K (E^{\tr_K})^{\tr_K}\to \Hom_{\C_{\dg}(K)}(E^{\tr_K},M\ten_A E^{\tr_K})
\]
and there is a chain map
\[
\xymatrix@R=0.4pc{
M\ten_A E^{\tr_K}\ar[r] & \Hom_{\C_{\dg}(A)}(M,E)^{\tr_K},\\
m\ten f \ar@{|->}[r] & (g\mapsto f(g(m)))
}
\]
which is an isomorphism for $M=A$ and hence a quasi-isomorphism for $M\in \per(A)$.
Therefore, for $M\in\per(A)$, we have
\begin{align*}
\tw^-_E(M)&=\mathrm{Cone}(\mathbf{p}M{\longrightarrow}\Hom_{\C_{\dg}(A)}(\mathbf{p}M,E)^{\tr_K}\ten_K E)[-1]\\
&=\mathrm{Cone}(M\longrightarrow \bigoplus_{p\in\mathbb{Z}}\Hom_{\D(A)}(M,E[p])^{\tr_K}\ten_K E[p])[-1].
\end{align*}

Let $X=\TW^-_E(A)$. Then $X$ has the structure of a dg $A$-$A$-bimodule. The following result is an immediate consequence of Lemma~\ref{lem:turning-a-derived-functor-into-tensor} because $\tw^-_E$ commutes with infinite direct sums.

\begin{lemma}\label{lem:twist-functor-as-tensor-functor}
There is an isomorphism of triangle functors 
$$-\mathop{{\otimes}^\mathbf{L}_A}X\cong \tw^-_E(-):\D(A)\to\D(A).$$
\end{lemma}

Let $d\geq 1$ be an integer. We say that $E$ is a \emph{$d$-spherical object} of $\D(A)$ if 
\begin{itemize}
\item[-] $\Hom_{\D(A)}(E,E[p])=\begin{cases} K & \text{if }p=0 \text{ or } d\\ 0 & \text{otherwise},\end{cases}$
\item[-] there is an isomorphism
\[
D\Hom_{\D(A)}(E,M)\cong \Hom_{\D(A)}(M,E[d])
\]
which is functorial in $M\in\D(A)$.
\end{itemize}

The following result follows from \cite[Theorem 5.1]{AL13}, see also \cite[Propsition 2.10]{ST01} and \cite[Lemma 3.1]{HKP16}.

\begin{proposition}\label{prop:twist-functor}
Assume that $E$ is a $d$-spherical object of $\D(A)$. Then $\tw^-_E$ is an autoequivalence of $\D(A)$ with quasi-inverse  the right derived functor $\tw_E$ of $\TW_E$.
\end{proposition}

%In the situation of Proposition~\ref{prop:twist-functor}, we have an isomorphism of triangle functors $\tw^-_E\cong -\mathop{{\otimes}^\mathbf{L}_A} \RHom_A(X,A)$ by \cite[Lemma 6.2]{Ke1}.
%%%%%%%%%%%%%%%%%%%%%%%%%%%%%%%%%%%%%%%%%%%%%%%%%%%%%%%%
%%%%%%%%%%%%%%%%%%%%%%%%%%%%%%%%%%%%%%%%%%%%%%%%%%%%%%%%
%%%%%%%%%%%%%%%%%%%%%%%%%%%%%%%%%%%%%%%%%%%%%%%%%%%%%%%%%%%%%%%%%%%%%%%%%%%%%%%%%%%%%%%%%%%%%%%%%

\subsection{Derived preprojective algebras}
\label{ss:derived-preprojective-algebra}
In this subsection we recall the definition of derived preprojective algebras from \cite{Ke2}.

\smallskip
Let $Q$ be a finite quiver. Define a graded quiver $\tilde{Q}$ by
\begin{itemize}
\item[-] $\tilde{Q}$ has the same vertices as $Q$;
\item[-] $\tilde{Q}$ has three types of arrows:
\begin{itemize}
\item[$\cdot$] the arrows of $Q$, in degree $0$,
\item[$\cdot$] $\alpha^*:j\to i$ in degree $0$, for each arrow $\alpha: i\to j$ of $Q$,
\item[$\cdot$] $t_i:i\to i$ in degree $-1$, for each vertex $i$ of $Q$.
\end{itemize}
\end{itemize}
The \emph{derived preprojective algebra} $\Gamma:=\Gamma(Q)$ is the dg algebra $(K\tilde{Q},d)$, where $K\tilde{Q}$ is the graded path algebra of $\tilde{Q}$ and $d$ is the unique $K$-linear differential which satisfies the graded Leibniz rule
\[
d(ab)=d(a)b+(-1)^{p}ad(b),
\]
where $a$ is homogeneous of degree $p$, and which takes the following values
\begin{itemize}
\item[$\cdot$] $d(e_i)=0$ for any vertex $i$ of $Q$, where $e_i$ is the trivial path at $i$,  
\item[$\cdot$] $d(\alpha)=0$ for any arrow $\alpha$ of $Q$,
\item[$\cdot$] $d(\alpha^*)=0$ for any arrow $\alpha^*$ of $Q$,
\item[$\cdot$] $d(t_i)=e_i\sum_\alpha (\alpha\alpha^*-\alpha^*\alpha)e_i$ for any vertex $i$ of $Q$, where in the summation $\alpha$ runs over all arrows of $Q$. 
\end{itemize}
Note that $H^{0}(\Gamma)$ is the ordinary preprojective algebra associated with $Q$. When $Q$ is Dynkin, this dg algebra appears in the derived McKay correspondence for Kleinian singularities, see for example \cite[Section 7.1]{KalckYang16}.

According to \cite[Theorem 6.3]{Ke2}, $\Gamma$ is homologically smooth and $2$-Calabi--Yau as a bimodule, \emph{i.e.}  $A\in\per(A^{op}\ten_K A)$ and there is an isomorphism $\RHom_{A^{op}\ten_K A}(A,A^{op}\ten_K A)\cong A[-2]$ in $\per(A^{op}\ten_K A)$. Thus, by \cite[Lemma 4.1]{Ke3}  and its proof, $\per(\Gamma)\supseteq \Dfd(\Gamma)$ and there is a functorial isomorphism for $X\in\Dfd(\Gamma)$ and $Y\in\D(\Gamma)$
\begin{align}
D\Hom(X,Y)\simeq \Hom(Y,X[2]). \label{eq:cy-property-of-derived-preprojective-algebra}
\end{align}
Note that since $\Gamma$ is concentrated in non-positive degrees, $\Gamma$ is a silting object of $\per(\Gamma)$ by Lemma~\ref{lem:non-positive-dg-algebra-and-silting}.

%%%%%%%%%%%%%%%%%%%%%%%%%%%%%%%%%%%%%%%%%%%%%%%%%%%%%%%%%%%%%%%%%%%%%%%%%%%%
\section{Main results}
\label{s:main-results}

Let $Q$ be an acyclic quiver and $\Gamma:=\Gamma(Q)$ be the derived preprojective algebra of $Q$ defined in Section~\ref{ss:derived-preprojective-algebra}. In this section we will study silting objects over $\Gamma$, especially we will prove Theorem~\ref{thm:main}.

\subsection{Twist functors and dg ideals}
\label{ss:twist-functors-vs-dg-ideals}

Let $i$ be a vetex of $Q$. Consider the dg ideal of $\Gamma$ (a dg ideal is a graded ideal closed under the differential)
\[
I_i:=I_i^+=\Gamma(1-e_i)\Gamma+\Gamma t_i\Gamma,
\] 
and let $I_i^-:=\Hom_{\C_{\dg}(\Gamma)}(I_i,\Gamma)$, which is a dg $\Gamma$-$\Gamma$-bimodule. Let $S_i=\Gamma/I_i$, which is 1-dimensional due to the fact that $Q$ has no loops, concentrated in degree $0$ and supported at $i$, and let $\tw^-_i=\tw^-_{S_i}:\D(\Gamma)\to\D(\Gamma)$ be the corresponding dual twist functor (Section~\ref{sss:twist-funtors}). 

\begin{theorem}\label{main1}
$S_i$ is $2$-spherical object of $\D(\Gamma)$. Moreover, 
there is an isomorphism of triangle autoequivalences of $\D(\Gamma)$
$$-\Lotimes I_i\cong\tw^-_i(-),$$
which restrict to isomorphic autoequivalences of $\per(\Gamma)$.
\end{theorem}

To prove Theorem \ref{main1} we need some preparation.

Put $R_i=\bigoplus_{\alpha\in Q_1:t(\alpha)=i} e_{s(\alpha)}\Gamma\oplus \bigoplus_{\beta\in Q_1:s(\beta)=i}e_{t(\beta)}\Gamma$, where $Q_1$ is the set of arrows of $Q$, and for an arrow $\rho$ we denote by $s(\rho)$ and $t(\rho)$ its source and target, respectively.  For $\alpha\in Q_1$ with $t(\alpha)=i$, let $\gamma_\alpha$ be the element $e_{s(\alpha)}$ in the direct summand $e_{s(\alpha)}\Gamma$ of $R_i$ corresponding to $\alpha$; and for $\beta\in Q_1$ with $s(\beta)=i$, let $\gamma_\beta$ be the element $e_{t(\beta)}$ in the direct summand $e_{t(\beta)}\Gamma$ of $R_i$ corresponding to $\beta$. Then 
\begin{align*}
f&:e_i\Gamma \to R_i, a\mapsto (\sum_{\alpha\in Q_1: t(\alpha)=i}\gamma_\alpha \alpha^*-\sum_{\beta\in Q_1:s(\beta)=i}\gamma_\beta \beta)a \text{ and }\\
g&:R_i  \to e_i\Gamma, \sum_{\alpha\in Q_1: t(\alpha)=i} \gamma_\alpha a_{\alpha}+\sum_{\beta\in Q_1:s(\beta)=i} \gamma_\beta b_\beta\mapsto \sum_{\alpha} \alpha a_{\alpha}+\sum_\beta \beta^*b_\beta
\end{align*}
are left $(\add((1-e_{i})\Ga))$-approximation and right $(\add((1-e_{i})\Ga))$-approximation of $e_i\Gamma$ in $\per(\Gamma)$, respectively. 
Define $t: e_i\Gamma[1]\to e_i\Gamma,~~a\mapsto t_ia$, which is homogeneous of degree $0$. We also freely view it as a morphism $e_i\Gamma[2]\to e_i\Gamma$ of degree $1$.

Let $L:=\mathrm{Cone}(\xymatrix@C20pt@R10pt{\hspace{-3pt}e_i\Gamma\ar[r]^{f}& R_i)}$. Then we have the following result. 

\begin{proposition}\label{exact seq}
There is an exact sequence of dg $\Gamma$-modules 
\[\xymatrix@C30pt@R10pt{
0 \ar[r]& L\ar[r]^{\left(\begin{smallmatrix}t\\g\end{smallmatrix}\right)} & e_i\Gamma\ \ar[r]&S_i\ar[r] &0, }\]
where the second homomorphism is the canonical projection from $e_i\Gamma$ onto $S_i$. 
As a consequence, the right dg $\Gamma$-module
\[
\mathbf{p}S_i=e_i\Gamma[2] \oplus R_i[1]\oplus e_i\Gamma
\]
with differential
\[
\begin{pmatrix}
d_{e_i\Gamma[2]} & 0 & 0\\
-f & d_{R_i[1]} & 0\\
t & g & d_{e_i\Gamma}
\end{pmatrix}
\]
is a $\KK$-projective resolution of $S_i$ over $\Gamma$. 
\end{proposition}
\begin{proof}
Let $L'$ be the kernel of the canonical projection from $e_i\Gamma$ to $S_i$. Then $L'$ is closed under the differential of $e_i\Gamma$ and as a graded $\Gamma$-submodule of $e_i\Gamma$ it is 
\[
L'=t_i\Gamma\oplus\bigoplus_{\alpha:t(\alpha)=i}\alpha\Gamma\oplus\bigoplus_{\beta:s(\beta)=i}\beta^*\Gamma.
\]
It is straightforward to check that the assignment $t_i\mapsto e_i$, $\alpha\mapsto e_{s(\alpha)}$, $\beta^*\mapsto e_{t(\beta)}$ extends to an isomorphism of right dg $\Gamma$-modules $L'\to L$. Moreover, the composition of this isomorphism with the homomorphism ${t\choose g}:L\to e_i\Gamma$ is exactly the inclusion $L'\to e_i\Gamma$.
 
Now the second statement follows from the first one since the given right dg module $\mathbf{p}S_i$ is exactly the mapping cone of $t\choose g$, which is $\KK$-projective since $\Gamma$ is $\KK$-projective and the property of being $\KK$-projective is closed under taking direct sums, direct summands, shifts and mapping cones.
\end{proof}

Now we are ready to give a proof of Theorem \ref{main1}.

\begin{proof}[Proof of  Theorem~\ref{main1}]
Direct computation using and Proposition \ref{exact seq}, the formula~\eqref{eq:Hom-space-in-derived-category} and the equality $\Hom_{\C_{\dg}(A)}(e_j\Gamma,M)=Me_j$ for any $j\in Q_0$ and any dg $\Gamma$-module $M$
shows that $$\Hom_{\D(A)}(S_i,S_i[p])=H^p\Hom_{\C_{\dg}(A)}(\mathbf{p}S_i,S_i)$$ is $K$ if $p=2$ and $p=0$, and is trivial otherwise. 
Therefore, by the relative Calabi--Yau property \eqref{eq:cy-property-of-derived-preprojective-algebra} we see that $S_i$ is a $2$-spherical object of $\D(\Gamma)$, so by Proposition~\ref{prop:twist-functor}, $\tw^-_i:\D(\Gamma)\to\D(\Gamma)$ is a triangle equivalence.
It restricts to an autoequivalence of $\per(\Gamma)$.

By the definition of $S_i$, there is a short exact sequence of dg $\Gamma$-$\Gamma$-bimodules
\[
0\to I_i\to \Gamma\to S_i\to 0.
\]
Therefore as a dg $\Gamma$-$\Gamma$-bimodule, $I_i$ is quasi-isomorphic to $X=\mathrm{Cone}(\Gamma\to S_i)[-1]$. So there is an isomorphism of triangle functors $-\Lotimes I_i\cong -\Lotimes X:\D(\Gamma)\to\D(\Gamma)$. 
Since $\TW^-_{S_i}(\Gamma)=\mathrm{Cone}(\Gamma\overset{\rm coev}{\longrightarrow}\Hom_{\C_{\dg}(K)}(S_i^{\tr_K},\Gamma\otimes_\Gamma S_i^{\tr_K}))[-1]=\mathrm{Cone}(\Gamma\to S_i)[-1]=X$, Lemma~\ref{lem:twist-functor-as-tensor-functor} implies $\tw^-_i(-)
\cong -\Lotimes X$. 
Thus, this gives the desired isomorphism of triangle functors.
\end{proof}

\subsection{Silting mutations and dg ideals}
\label{ss:silting-mutations-vs-dg-ideals}

There is a decomposition $\Gamma=\bigoplus_i e_i\Gamma$ of the right dg $\Gamma$-module $\Gamma$ into the direct sum of indecomposable direct summands, where $i$ runs over all vertices of $Q$. Let $i$ be a vertex of $Q$. 
Define the left mutation 
$$\mu_{i}^\Le(\Ga)=\mu_{e_i\Gamma}^\Le(\Ga)=\mathrm{Cone}(
{e_{i}\Ga}\stackrel{f}{\to}R_i)\oplus(1-e_{i})\Ga,$$ 
and the right mutation
$$\mu_{i}^\Ri(\Ga)=\mu_{e_i\Gamma}^\Ri(\Ga)=\mathrm{Cone}(
R_i\stackrel{g}{\to}{e_{i}\Ga})[-1]\oplus(1-e_{i})\Ga,$$ 
where $f$ and g are the $(\add((1-e_{i})\Ga))$-approximations of $e_i\Gamma$ defined in the beginning of Section~\ref{ss:twist-functors-vs-dg-ideals}.

\begin{lemma}\label{mutation and ideal}
As right dg $\Gamma$-modules, we have 
$$I_i\cong\mu_i^\Le(\Gamma) \text{ and } I_i^-\cong\mu_i^\Ri(\Gamma).$$ 
\end{lemma}

\begin{proof}
As a right dg $\Gamma$-module, we have 
$$I_i=e_iI_i\oplus (1-e_i)\Gamma.$$ 
On the other hand, there is a short exact sequnece of right dg $\Ga$-modules
$0\to e_iI_i\to e_i\Gamma\to S_i\to 0.$
Thus, Proposition \ref{exact seq} shows $e_iI_i\cong \mathrm{Cone}(\xymatrix@C20pt@R10pt{\hspace{-3pt}e_i\Gamma\ar[r]^{f}& R_i)}$ and we obtain the desired result for $I_i$. Similarly, there is an isomorhpism of left dg $\Gamma$-modules
\[
I_i\cong \mathrm{Cone}(\xymatrix@C20pt@R10pt{\hspace{-3pt}\Gamma e_i\ar[r]^{f'}& R'_i)}\oplus\Gamma(1-e_i),
\] 
where $R'_i=\bigoplus_{\alpha\in Q_1:t(\alpha)=i} \Gamma e_{s(\alpha)}\oplus \bigoplus_{\beta\in Q_1:s(\beta)=i}\Gamma e_{t(\beta)}$ and $f'$ is the homomorphism 
\[
f':\Gamma e_i \to R'_i, a\mapsto a(\sum_{\alpha\in Q_1:t(\alpha)=i}\alpha\gamma'_\alpha+\sum_{\beta\in Q_1:s(\beta)=i}\beta^*\gamma'_\beta).
\]
Here $\gamma'_\alpha$ is the element $e_{s(\alpha)}$ in the direct summand $\Gamma e_{s(\alpha)}$ of $R'_i$ corresponding to $\alpha\in Q_1$ with $t(\alpha)=i$ and $\gamma'_\beta$ is the element $e_{t(\beta)}$ in the direct summand $\Gamma e_{t(\beta)}$ of $R'_i$ corresponding to $\beta\in Q_1$ with $s(\beta)=i$. 
It is easily seen that $\Hom_{\C_{\dg}(\Gamma)}(f',\Gamma)\cong g$. Consequently,
\[
I_i^-=\Hom_{\C_{\dg}(\Gamma)}(I_i,\Gamma)\cong\mathrm{Cone}(
R_i\stackrel{g}{\to}{e_{i}\Ga})[-1]\oplus(1-e_{i})\Ga=\mu_{i}^\Ri(\Ga).\qedhere
\]
\end{proof}

Then we obtain the following corollary.

\begin{corollary}
\label{cor:projectivity-of-I}
\begin{itemize}
\item[(1)] $I_i$ and $I_i^-$ are $\KK$-projective both as right dg $\Gamma$-modules and as left dg $\Gamma$-modules.
\item[(2)] $-\Lotimes I_i^-$ is a quasi-inverse of $-\Lotimes I_i$. 
\item[(3)] There are isomorphisms of triangle functors $-\Lotimes I_i\cong -\ten_\Gamma I_i$ and $-\Lotimes I_i^-\cong -\ten_\Gamma I_i^-$. 
\item[(4)] There are isomorphisms $I_i^-\ten_\Gamma I_i\cong\Gamma\cong I_i\ten_\Gamma I_i^-$ in $\D(\Gamma^{op}\ten_K \Gamma)$.
\end{itemize}
\end{corollary}

\begin{proof}
(1) follows from Lemma \ref{mutation and ideal}. 
Then (2) and (3) follows from Lemma~\ref{lem:derived-tensor-functor}. 
For the first isomorphism in  
(4), note that the counit $(-\ten_\Gamma I_i^-)\ten_\Gamma I_i\to \mathrm{Id}_{\D(\Gamma)}$ of the adjoint pair $(-\ten_\Gamma I_i,-\ten_\Gamma I_i^-)$ is induced by the evaluation map $I_i^-\ten_\Gamma I_i\to\Gamma$, which is a quasi-isomorphism of dg bimodules because the counit is an isomorphism of triangle functors. For the second isomorphism, note that the unit is induced by the zigzag of dg bimodule homomorphisms
\[
\Gamma\to \Hom_{\C_{\dg}(\Gamma)}(I_i,I_i)\leftarrow I_i\ten_\Gamma I_i^-,
\]
where the left map is the structure map of $I_i$ and the right map is the canonical map taking $x\ten f$ to $(y\mapsto xf(y))$. This yields an isomorphism $\Gamma\cong I_i\ten_\Gamma I_i^-$ in $\D(\Gamma^{op}\ten_K \Gamma)$ because the unit is an isomorphism of triangle functors.
\end{proof}

The following corollary of Theorem~\ref{main1} is a non-complete version of a result in  \cite[Section III]{BIRS09} for complete preprojective algebras.

\begin{corollary}
\label{cor:BIRS}
Let $Q$ be a connected non-Dynkin quiver and $\Pi$ be the (ordinary) preprojective algebra of $Q$. For any vertex $i$ of $Q$, the ideal $\overline{I}_i:=\Pi(1-e_i)\Pi$ is a two-sided tilting complex of $\Pi$ and 
$-\mathop{{\otimes}^\mathbf{L}_\Pi}\nolimits\overline{I}_i$ gives an autoequivalence of $\D(\Pi)$.
\end{corollary}

\begin{proof}
By definition we have isomorphisms
\[
{}_{\Gamma}\Pi_{\Pi}\ten_{\Pi}{}_\Pi(\overline{I}_i)_{\Pi}\cong{}_\Gamma(\overline{I}_i)_{\Pi}\cong{}_{\Gamma}(I_i)_{\Gamma}\ten_\Gamma {}_{\Gamma}\Pi_{\Pi}.
\]
For a $\KK$-projective dg $\Gamma$-module $M$, the dg $\Pi$-module $M\ten_{\Gamma}{}_{\Gamma}\Pi_{\Pi}$ is $\KK$-projective and by Corollary~\ref{cor:projectivity-of-I}(1), the dg $\Gamma$-module $M\ten_{\Gamma}{}_{\Gamma}(I_i)_{\Gamma}$ is $\KK$-projective. Moreover, there is a canonical isomorphism of dg $\Pi$-modules
\[
M\ten_{\Gamma}{}_{\Gamma}\Pi_{\Pi}\ten_{\Pi}{}_\Pi(\overline{I}_i)_{\Pi}\cong M\ten_{\Gamma}{}_{\Gamma}(I_i)_{\Gamma}\ten_\Gamma {}_{\Gamma}\Pi_{\Pi}.
\]
This implies that there is a canonical isomorphism for $M\in\D(\Gamma)$
\[
(M\Lotimes {}_{\Gamma}\Pi_{\Pi})\Lotimespi {}_\Pi(\overline{I}_i)_{\Pi}\cong (M\Lotimes {}_{\Gamma}(I_i)_{\Gamma})\Lotimes {}_{\Gamma}\Pi_{\Pi},
\]
namely, there is a commutative diagram of triangle functors (up to isomorphism)
\[
\xymatrix@=3pc{
\D(\Gamma)\ar[r]^{-\Lotimes {}_{\Gamma}\Pi_{\Pi}} \ar[d]_{-\Lotimes I_i} & \D(\Pi)\ar[d]^{-\Lotimespi \overline{I}_i}\\
\D(\Gamma)\ar[r]^{-\Lotimes {}_{\Gamma}\Pi_{\Pi}} & \D(\Pi).
}
\]
Now by \cite[Section 4.2]{Ke3} (see also \cite[Lemma 9.1]{KalckYang18b}), the canonical projection $\Ga\to \Pi$ is quasi-isomorphism of dg algebras, so  $-\Lotimes {}_{\Gamma}\Pi_{\Pi}\colon\D(\Gamma)\to\D(\Pi)$ is a triangle equivalence. 
The desired result follows then from Theorem~\ref{main1}.
\end{proof}

\subsection{Silting objects and braid groups}
In this section, we study $\silt(\Gamma):=\silt\per(\Gamma)$.

Recall that there is a decomposition $\Gamma=\bigoplus_i e_i\Gamma$ of the right dg $\Gamma$-module $\Gamma$ into the direct sum of indecomposable direct summands, where $i$ runs over all vertices of $Q$. Below when we perform silting mutation we will keep this labeling of the direct summands and simply denote by $\mu^L_i$ the left mutation at the $i$-th direct summand. For example, in $\per(\Gamma)$ we have
$$\mu_{i}^\Le(\Ga)=\mathrm{Cone}(
{e_{i}\Ga}\stackrel{f}{\to}R_i)\oplus(1-e_{i})\Ga,$$ 
where $f$ is the left $(\add((1-e_{i})\Ga))$-approximation of $e_i\Gamma$ in $\per(\Gamma)$ defined in the beginning of Section~\ref{ss:twist-functors-vs-dg-ideals}.

Let $F_{Q}$ be the free group generated by variables $a_i$,  where $i$ runs over all vertices of $Q$. We denote by $a_i^-$ the inverse of $a_i$. 
For $a=a_{i_1}^{\epsilon_{1}}a_{i_2}^{\epsilon_2}\cdots a_{i_k}^{\epsilon_{k}}\in F_Q$, where $\epsilon_j\in\{+,-\}$, let
\begin{align*}
I_a&:=I_{i_1}^{\epsilon_{1}}\ten_\Gamma I_{i_2}^{\epsilon_{2}}\ten_\Gamma\cdots \ten_\Gamma I_{i_k}^{\epsilon_{k}}.
\end{align*} 
Let $B_Q$ be the braid group associated with $Q$, \emph{i.e.} $B_Q$ is the quotient group of $F_Q$ by the subgroup generated by the  elements $r_{ij}$, where $(i,j)$ runs over all pairs of different vertices of $Q$, and 
\[
r_{ij}=\begin{cases}a_ia_ja_i^{-}a_j^{-} & \text{if there is no edge between $i$ and $j$ in $Q$},\\
a_ia_ja_ia_j^-a_i^-a_j^- & \text{if there is an edge between $i$ and $j$ in $Q$}.
\end{cases}
\]

We denote by $B_{Q}^+$ the positive braid monoid, that is, $B_{Q}^+$ is the submonoid of $B_Q$ generated by $a_i$ for $i\in Q_0$. 
For $a,b$ of $B_Q$, we write $a\geq b$ if $ab^{-1}\in B_{Q}^+$, and it is a partial order on $B_Q$.We refer to \cite{KT} for the background of braid groups.

The aim of this subsection is to show the following result.

\begin{theorem}\label{main2}
\begin{itemize}
\item[(1)] 
There is an order-reversing map 
$$B_Q \to \silt(\Ga),~~a\mapsto I_a.$$
\item[(2)] If $Q$ is Dynkin, then the map in (1) is a bijection. 
\end{itemize}
\end{theorem}

\begin{remark}
\label{rem:group-structure-on-silt}
If $Q$ is Dynkin, by Theorem \ref{main2}, every silting object of $\per(\Gamma)$ is isomorphic to some $I_a$. Therefore $\silt(\Gamma)$ becomes a group with product
\[ I_a \cdot I_b:= I_a\ten_\Gamma I_b \]
and the map is a group isomorphism. 
\end{remark}

Put $\mu^+:=\mu^\Le$ and $\mu^-:=\mu^\Ri$. 
The next statement  shows that the left action of $I_i^{\epsilon}\ten_\Gamma-$ induces the mutation $\mu_i^{\epsilon}$ of silting objects in $\per(\Gamma)$. 
%$\mu_a(\Ga):=\mu_{i_1}^{\epsilon_{1}}\circ\mu_{i_2}^{\epsilon_2}\circ\cdots \circ \mu_{i_k}^{\epsilon_{k}}(\Ga),$

\begin{lemma}\label{I=mu}
For $a=a_{i_1}^{\epsilon_{1}}a_{i_2}^{\epsilon_2}\cdots a_{i_k}^{\epsilon_{k}}\in F_{Q}$, the object $I_{a}$ of $\per(\Gamma)$ is a silting object. Moreover, $I_a=I_{i_1}^{\epsilon_1}\ten_\Gamma I_{a'}$ is a mutation $\mu_{i_1}^{\epsilon_1}(I_{a'})$ of $I_{a'}$ for $a'\in a_{i_2}^{\epsilon_2}\cdots a_{i_k}^{\epsilon_{k}}$.
\end{lemma}

\begin{proof} By Corollary~\ref{cor:projectivity-of-I} (1), $I_a$ is $\KK$-projective both as right dg $\Gamma$-module and left dg $\Gamma$-module. Therefore there is a triangle functor
\[
-\ten_\Gamma I_a: \D(\Gamma)\longrightarrow \D(\Gamma),
\]
which is the composition of $-\ten_\Gamma I_{i_1}^{\epsilon_1},\ldots,-\ten_\Gamma I_{i_k}^{\epsilon_k}$. These latter triangle functors, by Theorem~\ref{main1}, are triangle equivalences. So $-\ten_\Gamma I_a$ is a triangle equivalence, and hence restricts to a triangle equivalence $\per(\Gamma)\to\per(\Gamma)$. Since $I_a$ is the image of the silting object $\Gamma$ of $\per(\Gamma)$ under this triangle equivalence, itself is a silting object.

For the second statement, consider the triangle equivalence 
\[
-\ten_\Gamma I_{a'}: \D(\Gamma)\to \D(\Gamma).
\]
It takes a left (respectively, right) approximation to a left (respectively, right) approximation, and hence takes the mutation $\mu_{i_1}^{\epsilon_1}(\Gamma)$ of $\Gamma$ to a mutation $\mu_{i_1}^{\epsilon_1}(I_{a'})$ of $I_{a'}$. By Lemma~\ref{mutation and ideal}, $\mu_{i_1}^{\epsilon_1}(\Gamma)\cong I_{i_1}^{\epsilon_1}$, it follows that $\mu_{i_1}^{\epsilon_1}(I_{a'})\cong I_{i_1}^{\epsilon_1}\ten_\Gamma I_{a'}=I_a$.
\end{proof}

In view of Lemma~\ref{I=mu}, for $a,b\in F_Q$ we put
\[
\mu_a(I_b)=I_a\ten_\Gamma I_b,
\]
which is obtained from $I_b$ by applying a sequence of silting mutations determined by $a$.

\medskip
Next we show that the $I_i$'s satisfy braid relations. 

\begin{lemma}\label{lem:braid-relations}
Let $i,j$ be two different vertices of $Q$. 
\begin{itemize}
\item[(1)] If there is no edge between $i$ and $j$, then $I_iI_j=I_jI_i$. 
\item[(2)] If there is precisely one edge between $i$ and $j$, then there is a dg ideal $I(i,j)$ of $\Gamma$ together with two quasi-isomorphisms of dg $\Gamma$-$\Gamma$-bimodules
\[
I_iI_jI_i\longrightarrow I(i,j)\longleftarrow I_jI_iI_j.
\]
\end{itemize}

\end{lemma}
\begin{proof}
The idea of this proof comes from the proof of \cite[Proposition III.1.8]{BIRS09}.

Assume that there is no edge between $i$ and $j$. Let $I(i,j)=\Gamma(1-e_i-e_j)\Gamma+\Gamma t_i\Gamma+\Gamma t_j\Gamma$, which is a dg ideal of $\Gamma$. Then $I(i,j)$ is contained in $I_iI_j$ and
\begin{align*}
I_iI_j
&\subseteq\Gamma(1-e_i)\Gamma(1-e_j)\Gamma+\Gamma t_i\Gamma+\Gamma t_j\Gamma\\
&\subseteq \Gamma(1-e_i-e_j)\Gamma+\Gamma e_j\Gamma e_i\Gamma+\Gamma t_i\Gamma+\Gamma t_j\Gamma.
\end{align*}
Recall from Section~\ref{ss:derived-preprojective-algebra} that as a graded $K$-algebra $\Gamma$ is the graded path algebra of the graded quiver $\tilde{Q}$. As a graded ideal of $\Gamma$, $\Gamma e_j\Gamma e_i\Gamma$ is generated by the paths of $\tilde{Q}$ from $i$ to $j$. Under our assumption such a path has to pass through a vertex different from $i$ and $j$, and hence belongs to $\Gamma(1-e_i-e_j)\Gamma$. Therefore
\[
I_iI_j\subseteq \Gamma(1-e_i-e_j)\Gamma+\Gamma t_i\Gamma+\Gamma t_j\Gamma=I(i,j).
\]
So $I_iI_j=I(i,j)$. 
Similarly, $I_jI_i=I(i,j)$ and hence $I_iI_j=I_jI_i$.

\smallskip
Assume that there is precisely one arrow from $i$ to $j$, say $\rho$. Define two dg ideals of $\Gamma$
\begin{align*}
I(i,j)&= \Gamma(1-e_i-e_j)\Gamma+\Gamma t_i\Gamma+\Gamma\rho^*\rho\Gamma+\Gamma t_j\Gamma+\Gamma \rho\rho^*\Gamma.\\
I_{i,j,i}&=\Gamma(1-e_i-e_j)\Gamma+\Gamma t_i^2\Gamma+\Gamma t_i\rho^*\rho\Gamma+\Gamma\rho^*\rho t_i\Gamma+\Gamma t_j\Gamma+\Gamma\rho\rho^*\Gamma.
\end{align*} 
Then $I_{i,j,i}$ is contained in $I_iI_jI_i$ and 
\begin{align*}
I_iI_jI_i&\subseteq \Gamma (1-e_i)\Gamma(1-e_j)\Gamma(1-e_i)\Gamma+\Gamma t_i\Gamma+\Gamma t_j\Gamma\\
&=\Gamma(1-e_i-e_j)\Gamma+\Gamma e_j\Gamma e_i\Gamma e_j\Gamma+\Gamma t_i\Gamma+\Gamma t_j\Gamma.
\end{align*}
As a graded ideal of $\Gamma$, $\Gamma e_j\Gamma e_i\Gamma e_j\Gamma$ is generated by the paths of $\tilde{Q}$ starting from $j$, passing through $i$ and ending in $j$. Under our assumption such a path either passes through a vertex different from $i$ and $j$, and hence belongs to $\Gamma(1-e_i-e_j)\Gamma$, or has a subpath $t_i$, and hence belongs to $\Gamma t_i \Gamma$, or has a subpath $\rho\rho^*$. Therefore
\[
I_iI_jI_i\subseteq \Gamma(1-e_i-e_j)\Gamma+\Gamma t_i\Gamma+\Gamma t_j\Gamma+\Gamma \rho\rho^*\Gamma\subseteq I(i,j).
\]
Thus we obtain a chain of dg ideals of $\Gamma$
\[
I_{i,j,i}\subseteq I_iI_jI_i\subseteq I(i,j).
\]
Moreover, the quotient $I(i,j)/I_{i,j,i}$ is the 2-dimensional contractible complex
\[
K\{t_i\}\oplus K\{\rho^*\rho\}
\]
with $t_i$ in degree $-1$ and $d(t_i)=-\rho^*\rho$. Since $\rho^*\rho$ does not belong to $I_iI_jI_i$, it follows that the quotient $I_iI_jI_i/I_{i,j,i}$ does not contain $\rho^*\rho$, which spans the unique proper subcomplex of $I(i,j)/I_{i,j,i}$. So $I_iI_jI_i/I_{i,j,i}$ is trivial and $I_iI_jI_i=I_{i,j,i}$. Therefore the inclusion $I_iI_jI_i\hookrightarrow I(i,j)$ is a quasi-isomorphism of dg $\Gamma$-$\Gamma$-modules. Similarly, $I_jI_iI_j$ is a dg subideal of $I(i,j)$ and the inclusion $I_jI_iI_j\hookrightarrow I(i,j)$ is a quasi-isomorphism.
\end{proof}

Then we give a proof of Theorem~\ref{main2} (1).

\begin{proof}[Proof of  Theorem~\ref{main2} (1)]
By Lemmas~\ref{mutation and ideal}~and~\ref{lem:braid-relations}, there is a map 
$$B_Q \ni a\mapsto I_a\in\silt(\Ga).$$ 
Now for two elements $a\geq b$ of $B_Q$, it follows from Corollary~\ref{cor:projectivity-of-I} (4) that 
$I_a$ is isomorphic in $\per(\Gamma)$ to $\mu_{ab^{-1}}(I_b)=I_{ab^{-1}}\ten_\Gamma I_b$, which is no greater than $I_b$ since $ab^{-1}\in B_Q^+$ and hence $\mu_{ab^{-1}}$ is the composition of a sequence of left mutations.
\end{proof}

Next we assume that $Q$ is Dynkin and give a proof  Theorem~\ref{main2} (2), which is similar to the proofs of \cite[Proposition 6.5]{AM15} and of \cite[Theorem 3.1]{BT11}. We need the following result, which is part of \cite[Lemma 8.5 and Corollary 8.6]{AMY}.

\begin{lemma}
\label{lem:Hom-finiteness}\label{silt-dic}
The category $\per(\Gamma)$ is Hom-finite and Krull--Schmidt. Moreover, it is silting-discrete, that is, 
$$
\{ M\in \silt(\Gamma) \mid \Ga\ge M \ge \Ga[n]\}
$$ 
is a finite set for any positive integer $n$. 
\end{lemma}

\begin{proof}[Proof of  Theorem~\ref{main2} (2)] 
We need to show that the map is bijective. 
By Lemma~\ref{silt-dic}, $\per(\Gamma)$ is silting-discrete and hence any silting object of $\per(\Gamma)$ can be obtained from $\Gamma$ by a sequence of mutations, by \cite[Corollary 3.9]{Ai13}. Therefore
the map is surjective. 

It remains to show the injectivity. By \cite[Lemma 2.3]{BT11}, it suffices to show that the restriction 
$B_{Q}^+\to\silt\Ga,\ a\mapsto I_a$, is injective. 

For  $a\in B_{Q}^+$ we denote by $\ell(a)$ its length, that is, the number of elements in a reduced expression of $a$. Take $b,c\in B_{Q}^+$ such that $I_b\cong I_c$ in $\per(\Ga)$. 
Without loss of generality, we may assume that $\ell(b)\le\ell(c)$. 
We show that $b=c$ by induction on $\ell(b)$.

If $\ell(b)=0$, \emph{i.e.} $b=\id$, then $I_b=\Ga$. 
Then we have $c=\id$ because otherwise $\Ga>I_c=\mu_c(\Ga)$ by \cite[Theorem 2.35]{AI12}, where $>$ is the partial order defined in Section \ref{sectionsilting}. 

Next assume that $\ell(b)>0$. 
We write 
$b=b'a_i$ and $c=c'a_j$ for some $b',c'\in B_{Q}^+$ and vertices $i,j$ of $Q$. 
If $i=j$, then $I_{b'}\cong I_{c'}$ and the induction 
hypothesis implies that $b'=c'$ and hence $b=c$. In the rest we assume that $i\neq j$ and define 
\[a_{i,j}:=
\left\{\begin{array}{ll}
\ a_ia_j & \mbox{if there is no edge between $i$ and $j$ in $Q$},\\
\ a_ia_ja_i & \mbox{if there is precisely one edge $i\stackrel{ }{\mbox{---}}j$ in $Q$}.
\end{array}\right.\]
Then $a_{i,j}$ is the unique minimal common upper bound (that is, join in $B_Q^+$) of $a_i$ and $a_j$ \cite[Theorem 6.19, 6.20, (Section 6.5.3)]{KT}. 
%\new{By \cite[Proposition A.3,A.6]{BY13}, there is a bijection between two-term silting objects over derived preprojective algebra of Dynkin and the one over (ordinary) preprojective algebra of Dynkin and their mutation structures are compatible. Then  \cite[Theorem 4.1]{AM15} (\cite[Theorem 2.30]{M14}) implies that $I_{a_{i,j}}$ is a greatest lower bound of $I_{i}$ and $I_{j}$ in the poset $(\silt(\Gamma),\geq)$.} \old{Then it follows from Lemma \ref{lem:braid-relations} and \cite[Theorem 2.35]{AI12} that $I_{a_{i,j}}$ is a meet of $I_{i}$ and $I_{j}$ in the poset $(\silt(\Gamma),\geq)$.}
Therefore, the fact that the map is order-reversing implies 
 $I_{a_{i,j}}\geq I_b\cong I_c$. %because $I_{i}\geq I_b$ and $I_{j}\geq I_c\cong I_b$. 
 
Applying \cite[Proposition 2.36]{AI12}, there exists a vertex $k_1$ of $Q$ such that $I_{a_{i,j}}> \mu_{k_1}^\Le(I_{a_{i,j}})\geq I_b$. Repeating this we obtain a sequence of silting objects 
\[
I_{a_{i,j}}\gneq\mu_{k_1}^\Le(I_{a_{i,j}})\gneq\mu_{k_2}^\Le\circ\mu_{k_1}^\Le(I_{a_{i,j}})\gneq\ldots\geq I_b.
\]

Because $\per(\Gamma)$ is silting-discrete, 
there are only finitely many isomorphism classes of basic silting objects 
between $I_{a_{i,j}}$ and $I_b$. Therefore the above sequence has to be finite. Namely, there exists $d\in B_{Q}^+$ such that 
$I_d\Lotimes I_{a_{i,j}}= I_{da_{i,j}}\cong I_b\cong I_{b'}\Lotimes I_i$. It follows that $I_{da_{i,j}a_i^{-1}}\cong I_{b'}$. 
Since $da_{i,j}a_i^{-1}\in B_{Q}^+$, the induction
hypothesis implies that $da_{i,j}a_i^{-1}=b'$ and hence $da_{i,j}=b$. 
Similarly, we have $I_{da_{i,j}a_j^{-1}}\cong I_{c'}$ and  
$da_{i,j}a_j^{-1}=c'$. 
Therefore $b=da_{i,j}=c'a_j=c$ and the proof is complete.
\end{proof}

As a consequence of Corollary~\ref{cor:projectivity-of-I}, Theorem~\ref{main2} and Lemma~\ref{lem:braid-relations}, we have

\begin{corollary}
\label{cor:braid-group-action}
The braid group $B_Q$ acts on $\D(\Gamma)$ with $a_i$ acting as $\tw^-_i$ ($i\in Q_0$). When $Q$ is Dynkin, this action is faithful.
\end{corollary}

\subsection{Spherical collections}
\label{ss:spherical-collection}
As a consequence of the previous results, we give a classification of spherical collections.

A collection $(E_1,\ldots,E_r)$ of objects of $\D_{\fd}(\Gamma)$ is called a \emph{spherical collection} if 
\begin{itemize}
\item[-] all $E_1,\ldots,E_r$ are $2$-spherical objects,
\item[-] $\Hom(E_i, E_j[p])=0$ for all $i\neq j$ and for all $p\leq 0$,
\item[-] $\D_{\fd}(\Gamma)=\thick(E_1,\ldots,E_r)$.
\end{itemize} 
For example, the collection $\mathcal{S}=(S_i)_{i\in Q_0}$ of simple dg $\Gamma$-modules is a spherical collection. Moreover, the property of being a spherical collection is preserved under applying autoequivalences. 

\smallskip
A classification of spherical objects in $\D_{\fd}(\Gamma)$ is given in \cite{BapatDeopurkarLicata23,HaraWemyss22}:  any spherical object can be obtained from some $S_k$ by applying a sequence of spherical twist functors $\tw^{\pm}_i$ ($i\in Q_0$). Part (b) of the following result can also be proved by using \cite[Corollaries 1.2 and 1.4]{HaraWemyss22}.

\begin{theorem}
\label{thm:spherical-collection}
Assume that $Q$ is Dynkin.
\begin{itemize}
\item[(a)] Any basic silting object of $\per(\Gamma)$ can be obtained from $\Gamma$ by applying a sequence of spherical twist functors $\tw^{\pm}_i$ ($i\in Q_0$).
\item[(b)] Any spherical collection of $\D_{\fd}(\Gamma)$ can be obtained from $\mathcal{S}$ by applying a sequence of spherical twist functors $\tw^{\pm}_i$ ($i\in Q_0$).
\end{itemize}
\end{theorem}
\begin{proof}
(a) This is a consequence of 
Theorem~\ref{main1}, Corollary~\ref{cor:projectivity-of-I} and Theorem~\ref{main2}(2).

(b) We first remark that a spherical collection is a simple-minded collection. By \cite{KN2} (see \cite[Theorem 6.3]{N} and see also \cite[Theorem B]{Fu23} for a more recent and more general result), there is a bijection between the set of isoclasses of basic silting objects of $\per(\Gamma)$ and the set of simple-minded collections of $\D_{\fd}(\Gamma)$. This bijection takes $\Gamma$ to $\mathcal{S}$. Moreover, according to \cite[Remark 3.7]{Fu23}, if a basic silting object $M=M_1\oplus\ldots\oplus M_r$ with $M_1,\ldots,M_r$ indecomposable and a simple-minded collection $(X_1,\ldots,X_r)$ correspond to each other under this bijection, then they determine each other by the following Hom-duality up to reordering:
\[
\Hom(M_i,\Sigma^p X_j)=\begin{cases} k & \text{if } i=j \text{ and }p=0,\\
0 & \text{otherwise.}
\end{cases}
\]
It follows that for any autoequivalence $F$ of $\D(\Gamma)$, $F(M)$ and $F(X_1,\ldots,X_r)$ also correspond to each other under the bijection.
Thus it follows from (a) that any simple-minded collection of $\D_{\fd}(\Gamma)$ can be obtained from $\mathcal{S}$ by a sequence of spherical twist functors $\tw^{\pm}_i$ ($i\in Q_0$). In particular, any simple-minded collection is a spherical collection and the desired result follows.
\end{proof}

%%%%%%%%%%%%%%%%%%%%%%%%%%%%%%%%%%%%%%%%%%%%%%%%%%%%%%%%


\begin{thebibliography}{99}
\bibitem[AMY]{AMY} T.~Adachi, Y.~Mizuno, D.~Yang, 
\emph{Discreteness of silting objects and t-structures in triangulated categories}, Proc. Lond. Math. Soc.  
\textbf{118} (1) 2019, 1--42.


\bibitem[A]{Ai13} T.~Aihara, \emph{Tilting-connected symmetric algebras}, Algebr. Represent. Theory \textbf{16} (2013), no.~3, 873--894.

\bibitem[AI]{AI12} T.~Aihara, O.~Iyama, \emph{Silting mutation in triangulated categories}, J. London. Math. Soc. (2) \textbf{85} (2012), no.~3, 633--668.

\bibitem[AM]{AM15} T.~Aihara, Y.~Mizuno, \emph{Classifying tilting complexes over preprojective algebras of Dynkin type}, Algebra Number Theory \textbf{11} (2017), no. 6, 1287--1315.

%\bibitem[AHMV]{AHMV} L.~Angeleri H\"ugel, F.~Marks, J.~Vit\'oria, \emph{Silting modules}, Int. Math. Res. Not. IMRN \textbf{2016}, no.~4, 1251--1284.

%\bibitem[AHMV2]{AHMV2} L.~Angeleri H\"ugel, F.~Marks, J.~Vit\'oria, \emph{Silting modules and ring epimorphisms}, Adv. Math. \textbf{303} (2016), 1044--1076.

%\bibitem[An]{An07} R.~Anno, \emph{Spherical functors}, arXiv:0711.4409v5.

\bibitem[AL]{AL13} R.~Anno, T.~Logvinenko, \emph{Spherical DG-functors}, J. Eur. Math. Soc. (JEMS) 19 (2017), no. 9, 2577--2656.

%\bibitem[AS]{AS80} M.~Auslander, S.~O.~Smal{\o}, \emph{Preprojective modules over Artin algebras}, \new{Journal of Algebra, \textbf{66}, no.~1 , 1980, 61--122}.

\bibitem[BDL]{BapatDeopurkarLicata23} 
A.~Bapat, A.~Deopurkar and A.~Licata, \emph{Spherical objects and stability conditions on 2-Calabi--Yau quiver categories}, Math. Z. \textbf{303}, no. 1, Paper No. 13, 22 pp.

%\bibitem[BBD]{BBD81} A.~A.~Beilinson, J.~Bernstein, P.~Deligne, \emph{Faisceaux pervers}, Analysis and topology on singular spaces, I (Luminy, 1981), 5--171, Ast$\acute{\mathrm{e}}$risque, \textbf{100}, Soc. Math. France, Paris, 1982. 

\bibitem[BIRS]{BIRS09} A.~B.~Buan, O.~Iyama, I.~Reiten, J.~Scott, \emph{Cluster structures for 2-Calabi--Yau categories and unipotent groups}, Compos. Math. \textbf{145} (2009), no.~4, 1035--1079.

%\bibitem[BR]{BR07} A.~Beligiannis, I.~Reiten, \emph{Homological and homotopical aspects of torsion theories}, Mem. Amer. Math. Soc. \textbf{188} (2007), no.~883.

%\bibitem[Bo]{Bo10} \new{M.~V.~Bondarko}, \emph{Weight structures vs. {$t$}-structures; weight filtrations, spectral sequences, and complexes (for motives and in general)}, J. K-Theory \textbf{6} (2010), no.~3, 387--504.

\bibitem[BT]{BT11} C.~Brav, H.~Thomas, \emph{Braid groups and Kleinian singularities}, Math. Ann. \textbf{351} (2011), no.~4, 1005--1017.

%\bibitem[Br1]{Br1} T.~Bridgeland, \emph{Stability conditions on triangulated categories}, Ann. Math. (2), \textbf{166} (2007), no.~2, 317--345.

%\bibitem[Br2]{Br2} T.~Bridgeland, \emph{$t$-structures on some local Calabi--Yau varieties}, J. Algebra \textbf{289} (2005), no.~2, 453--483.

%\bibitem[BPP1]{BPP1} N.~Broomhead, D.~Pauksztello, D.~Ploog, \emph{Discrete derived categories I: homomorphisms, autoequivalences, t-structures}, Math. Z. \textbf{285} (2017), no.~1-2, 39--89.%arXiv:1312.5203.

%\bibitem[BPP2]{BPP2}  N.~Broomhead, D.~Pauksztello, D.~Ploog, \emph{Discrete derived categories II: The silting pairs CW complexes and the stability manifold}, J. Lond. Math. Soc. (2) \textbf{93} (2016), no.~2, 273--300.%arXiv:1407.5944.

\bibitem[BY]{BY13} T.~Br\"ustle, D.~Yang, \emph{Ordered exchange graphs. Advances in representation theory of algebras}, 135--193, EMS Ser. Congr. Rep., Eur. Math. Soc., Zurich, 2013.

\bibitem[BRT]{BRT11} A.~B.~Buan, I.~Reiten, H.~Thomas, \emph{Three kinds of mutations}, J. Algebra \textbf{339} (2011), 97--113.

%\old{\bibitem[BZ]{BZ16} A.~B.~Buan, Y.~Zhou, \emph{A silting theorem}, J. Pure Appl. Algebra \textbf{220} (2016), no. 7, 2748--2770.}%arXiv:1503.06129.

%\bibitem[DIJ]{DIJ15} L.~Demonet, O.~Iyama, G.~Jasso, \emph{$\tau$-tilting finite algebras, \new{bricks} and $g$-vectors}, arXiv:1503.00285v6, \new{to appear in Int. Math. Res. Not.}.

%\bibitem[DWZ]{DWZ08} H.~Derksen, J.~Weyman, A.~Zelevinsky, \emph{Quivers with potentials and their representations I: Mutations}, Selecta Mathematica \textbf{14} (2008), \new{no.~1}, 59--119.

%\bibitem[EJR]{EJR16} F.~Eisele, G.~Janssens, T.~Raedschelders, \emph{A reduction theorem for $\tau$-rigid modules}, arXiv:1603.04293v1.

%\bibitem[FZ]{FZ03} S.~Fomin, A.~Zelevinsky, \emph{Cluster algebras II. Finite type classification}, Invent. Math. \textbf{154} (2003), no.~1, 63--121.

\bibitem[F]{Fu23}
R.~Fushimi, \emph{The correspondence between silting objects and t-structures for non-positive dg algebras}, arXiv:2312.17597v2.

%\bibitem[Gi]{Gi06} V.~Ginzburg, \emph{Calabi-Yau algebras}, arXiv:math/0612139v3.

%\bibitem[Gu]{Gu11} L.~Guo, \emph{Cluster tilting objects in generalized higher cluster categories}, J. Pure Appl. Algebra \textbf{215} (2011), no.~9, 2055--2071.

%\bibitem[HRS]{HRS96} D.~Happel, I.~Reiten, S.~O.~Smal\o, \emph{Tilting in abelian categories and quasitilted algebras}, Mem. Amer. Math. Soc. {\bf 120} (1996), no. 575.


\bibitem[GLS1]{GLS1}C.~Geiss, B.~Leclerc, J.~Schr{\"o}er, \emph{Rigid modules over preprojective algebras},  Invent. Math. 165 (2006), no. 3, 589--632.


\bibitem[GLS2]{GLS2}C.~Geiss, B.~Leclerc, J.~Schr{\"o}er, \emph{Kac-Moody groups and cluster algebras}, Adv. Math. 228 (2011), no. 1, 329--433.

\bibitem[HW]{HaraWemyss22}
W.~Hara and M.~Wemyss, \emph{Spherical objects in dimension two and three}, arXiv:2205.11552.

\bibitem[HKP]{HKP16} A.~Hochenegger, M.~Kalck, D.~Ploog, \emph{Spherical subcategories in algebraic geometry}, Math. Nachr. \textbf{289} (2016), no.~11-12, 1450--1465.

%\bibitem[HKM]{HKM02} M.~Hoshino, Y.~Kato, J.~Miyachi, \emph{On $t$-structure and torsion theories induced by compact objects},  J. Pure Appl. Algebra \textbf{167} (2002), no. 1, 15--35. 

%\bibitem[I]{I17} A.~Ikeda, \emph{Stability conditions on $CY_N$ categories associated to $A_n$-quivers and period maps}, Math. Ann. \textbf{367} (2017), no.~1-2, 1--49.%arXiv:1405.5492v2.

\bibitem[IJY]{IJY14} O.~Iyama, P.~J{\o}rgensen, D.~Yang, \emph{Intermediate co-$t$-structures, two-term silting objects, $\tau$-tilting modules and torsion classes}, Algebra Number Theory \textbf{8} (2014), no.~10, 2413--2431.


\bibitem[IR]{IR}
O.~Iyama, I.~Reiten, \emph{Fomin-{Z}elevinsky mutation and tilting modules
  over {C}alabi-{Y}au algebras}, Amer. J. Math. 130 (2008), no.~4,
  1087--1149.


\bibitem[IY]{IYa14} O.~Iyama, D.~Yang, \emph{Silting reduction and Calabi--Yau reduction of triangulated categories}, Trans. Amer. Math. Soc. 370 (2018), no. 11, 7861--7898.

%\old{\bibitem[IYo]{IYo08} O.~Iyama, Y.~Yoshino, \emph{Mutation in triangulated categories and rigid Cohen--Macaulay modules}, Invent. Math. \textbf{172} (2008), no.~1, 117--168.}

\bibitem[KaY1]{KalckYang16} M.~Kalck, D.~Yang, \emph{Relative singularity categories I: Auslander resolutions}, Adv. Math. 301 (2016), 973--1021.

\bibitem[KaY2]{KalckYang18b} M.~Kalck, D.~Yang, \emph{Relative singularity categories II: DG models}, arXiv:1803.08192v1.

\bibitem[KT]{KT}  C.~Kassel, V.~Turaev, \emph{Braid groups}, Graduate Texts in Mathematics, 247. Springer, New York, 2008.


\bibitem[K1]{Ke1} B.~Keller, \emph{Deriving DG categories}, Ann. Sci. Ecole Norm. Sup. (4) \textbf{27} (1994), no.~1, 63--102. 

\bibitem[K2]{Ke2} B.~Keller, \emph{Deformed Calabi-Yau completions. With an appendix by Michel Van den Bergh}, J. Reine Angew. Math. \textbf{654} (2011), 125--180.

\bibitem[K3]{Ke3} B.~Keller, \emph{Triangulated Calabi--Yau categories}, Trends in Representation Theory of Algebras (Zurich) (A.~Skowro\'nski, ed.), European Mathematical Society, 2008, 467--489.
  
\bibitem[K4]{Ke4} B.~Keller, \emph{On differential graded categories}, International Congress of Mathematicians. Vol. II, Eur. Math. Soc., Z\"urich, 2006, 151--190.

\bibitem[KN1]{KN1} B.~Keller, P.~Nicol{\'a}s, \emph{Weight structures and simple dg modules for positive dg algebras}, Int. Math. Res. Not. IMRN  2013, no.~5, 1028--1078.

\bibitem[KN2]{KN2} B.~Keller, P.~Nicol{\'a}s, \emph{Cluster hearts and cluster tilting objects}, work in preparation.


%\old{\bibitem[KR]{KellerReiten08} B.~Keller, I.~Reiten, \emph{Acyclic Calabi--Yau categories}, Compos. Math. \textbf{144} (2008), 1332--1348.}

%\bibitem[KV]{KV88} B.~Keller, D.~Vossieck, \emph{Aisles in derived categories}, Bull. Soc. Math. Belg. S$\acute{\mathrm{e}}$r. A \textbf{40} (1988), no.~2, 239--253. 

\bibitem[KeY]{KeY11} B.~Keller, D.~Yang, \emph{Derived equivalences from mutations of quivers with potential}, Adv. Math. \textbf{226} (2011), no.~3, 2118--2168. With an appendix by Bernhard Keller.

\bibitem[KoY]{KoY14} S.~Koenig, D.~Yang, \emph{Silting objects, simple-minded collections, $t$-structures and co-$t$-structures for finite-dimensional algebras}, Doc. Math. \textbf{19} (2014), 403--438.

%\bibitem[L]{L01} T.~Y.~Lam, \emph{A first course in noncommutative rings}, second edition. Graduate Texts in Mathematics \textbf{131}, Springer-Verlag, New York, 2001.

%\bibitem[MS]{MS16} F.~Marks, J.~Stovicek, \emph{Universal localisations via silting}, arXiv:1605.04222v2.

%\bibitem[MSSS]{MSSS13} O.~Mendoza, E.~C.~S\'aenz, V.~Santiago, M.~J.~Souto Salorio, \emph{Auslander-buchweitz context and co-$t$-structures}, Appl. Categ. Structures \textbf{21} (2013), no.~5, 417--440.

\bibitem[M1]{M14} Y.~Mizuno, \emph{Classifying $\tau$-tilting modules over preprojective algebras of Dynkin type}, Math. Z. \textbf{277} (2014), no.~3-4, 665--690.

\bibitem[M2]{M21} Y.~Mizuno, \emph{Derived Picard groups of preprojective algebras of Dynkin type}, 
Int. Math. Res. Not. IMRN(2021), no. 7, 5111-5154.

\bibitem[N]{N} P. Nicol\'as, Notes of the talk ``Cluster-hearts and cluster-tilting objects" in Stuttgart, available as https://pnp.mathematik.uni-stuttgart.de/iaz/iaz1/activities/t-workshop/NicolasNotes.pdf.

%\bibitem[P]{P08} Y.~Palu, \emph{Cluster characters for 2-Calabi--Yau triangulated categories}, Ann. Inst. Fourier (Grenoble) \textbf{58} (2008), no.~6, 2221--2248.

%\bibitem[PSZ]{PSZ17}, D.~Pauksztello, M.~Saor\'in, A.~Zvonareva, \emph{Contractibility of the stability manifold for silting-discrete algebras}, arXiv:1705.10604v1.

%\old{\bibitem[Pl]{Pl} P.~Plamondon, \emph{Cluster algebras via cluster categories with infinite-dimensional morphism spaces}, Compos. Math. \textbf{147} (2011), no.~6, 1921--1954.}

%\bibitem[PV]{PV15} C.~Psaroudakis, J.~Vit\'oria, \emph{Realisation functors in tilting theory}, arXiv:1511.02677v2.

%\bibitem[Q]{Q15} Y.~Qiu, \emph{Stability conditions and quantum dilogarithm identities for Dynkin quivers}, Adv. Math. \textbf{269} (2015), 220--264.

%\bibitem[QW]{QW14} Y.~Qiu, J.~Woolf, \emph{Contractible stability spaces and faithful braid group actions}, arXiv:1407.5986v2.

\bibitem[ST]{ST01} P.~Seidel, R.~Thomas, \emph{Braid group actions on derived categories of coherent sheaves}, Duke Math. J. \textbf{108} (2001), no.~1, 37--108.


%\bibitem[SY]{SY16} H.~Su, D.~Yang, \emph{From simple-minded collections to silting objects via Koszul duality}, arXiv:1609.03767v3.

%\bibitem[W]{W10} J.~Woolf, \emph{Stability conditions, torsion theories and tilting}, J. Lond. Math. Soc. (2) \textbf{82} (2010), no.~3, 663--682.

%\bibitem[Y]{Y09} D.~Yang, \emph{From triangulated categories to cluster algebras (after Palu)}, talk given in Bielefeld Representation Seminar, July 2009. Talk notes available as http://maths.nju.edu.cn/$\sim$dyang/mytalk/\\TriCatCluAlg.pdf.
  
%\bibitem[Z]{Z08} B.~Zhu, \emph{Generalized cluster complexes via quiver representations}, J. Algebr. Comb. \textbf{27} (2008), \new{no.~1}, 35--54.
\end{thebibliography}
\end{document}